\documentclass[11pt,reqno]{amsart}

\usepackage[margin=1in]{geometry}
\usepackage{amsmath}
\usepackage{amssymb}
\usepackage{amsthm}
\usepackage{dsfont}
\usepackage{graphicx} 
\usepackage{microtype}

\usepackage{tikz}
\usetikzlibrary{shapes.geometric, positioning, fit}

\usepackage{charter}
\usepackage[cal=euler]{mathalpha}

\usepackage{xcolor}
\usepackage[backref=page]{hyperref}
\hypersetup{
    colorlinks=true,
    linkcolor=cyan!80!black,
    citecolor=cyan!80!black,
    urlcolor=magenta,
}

\usepackage{booktabs}
\usepackage{makecell}
\usepackage[symbol]{footmisc}

\usepackage{apptools}
\AtAppendix{\counterwithin{lemma}{section}}

\allowdisplaybreaks

\newcommand{\bbE}{\mathbb{E}}

\newcommand{\bbN}{\mathbb{N}}
\newcommand{\bbP}{\mathbb{P}}
\newcommand{\bbR}{\mathbb{R}}

\newcommand{\bbI}{\mathbb{I}}

\newcommand{\cN}{\mathcal{N}}

\newcommand{\cS}{\mathcal{S}}

\newcommand{\Unif}{\mathsf{Uniform}}

\renewcommand{\sc}{\mathsf{sc}}

\newcommand{\iid}{\mathrm{i.i.d.}}

\DeclareMathOperator{\Var}{Var}
\DeclareMathOperator{\Tr}{Tr}

\DeclareMathOperator{\Cov}{Cov}

\let\tilde\widetilde

\theoremstyle{plain}
\newtheorem{theorem}{Theorem}[section]
\newtheorem{corollary}{Corollary}[section]

\theoremstyle{definition}

\newtheorem{definition}{Definition}[section]

\theoremstyle{remark}

\title[Semicircular law with a few independent entries]{Semicircular law with a few independent entries in a random matrix}
\author[D. Banerjee]{Debapratim Banerjee}
\author[H. Talukdar]{Himasish Talukdar}
\address{
    Department of Mathematics\\
    Ashoka University, India.
}
\address{
    Statistics and Mathematics Unit \\
    Indian Statistical Institute \\
    203 B.T. Road, Kolkata 700108 \\
    West Bengal, India.
}
\email{debapratim.banerjee@ashoka.edu.in}
\email{talukdar.himasish@gmail.com}

\begin{document}

\begin{abstract}
It is well known in random matrix literature that the limiting spectral distribution of a Wigner matrix is the semi circular law while the limiting spectral distributions of other patterned matrices like Toeplitz, Hankel, symmetric circulant and reverse circulant matrices have unbounded supports. One fundamental difference between the Wigner matrices and the other matrices mentioned above is that the Wigner matrices have $O(n^2)$ independent random variables while the others have $O(n)$ independent random variables. In this paper, we show that this is not true in general. In particular, we form matrices with $O(n)$ independent random variables whose empirical spectral distributions are arbitrary close to the semi circular law.
\end{abstract}

\keywords{Wigner's Semicircle law, patterned random matrices, permutation matrices}

\maketitle
\thispagestyle{empty}

\section{Introduction}
The amazing journey of studying random matrix theory began with the famous Wigner's semicircular law in the seminal papers of Wigner \cite{wig1} and \cite{wig2}. Essentially, the result states that if someone takes an $n \times n$ matrix $W= \frac{1}{\sqrt{n}}\left( x_{i,j} \right)_{1\le i,j \le n}$ such that the upper diagonal part of $W$ is formed by i.i.d. mean $0$ variance $1$ random variables with all moments finite, then the empirical distribution of the eigenvalues of $W$ converges weakly to the semi circular law in almost sure sense. Here, the semicircular law is a probability distribution supported on $[-2,2]$ having the following density:
\begin{equation}\label{eq:semicircular}
 f(x)=\left\{
 \begin{array}{ll}
 \frac{1}{2\pi} \sqrt{4-x^2} & \text{when~ $|x|\le 2$}\\
 0 & \text{otherwise.} 
 \end{array}
 \right.
\end{equation}
Later, various patterned matrices such as the Toeplitz and Hankel matrices received significant interest in the community. A Toeplitz matrix $T=\frac{1}{\sqrt{n}}(t_{i,j})_{1\le i,j \le n}$ is an $n \times n$ matrix where $t_{i,j}= x_{|i-j|}$. Here, $x_{0},\ldots , x_{n-1}$ is a collection of $n$ i.i.d. mean $0$ variance $1$ random variables with all moments finite. Similarly, a Hankel matrix $H= \frac{1}{\sqrt{n}}\left( h_{i,j} \right)_{1\le i,j \le n}$ is an $n \times n$ matrix where $h_{i,j}= x_{i+j}$. Here $x_{2},\ldots , x_{2n}$ is a collection of $2n-1$ i.i.d. mean $0$ variance $1$ random variables with all moments finite. The study of such matrices is significantly harder than the Wigner matrices, as they are formed by fewer independent random variables. To the best of our knowledge, the explicit formulas of the limit of the empirical spectral distribution of these matrices are not yet known. However, it is known that the limit exists and is supported on the entire real line unlike the Wigner matrix. In addition to these two matrices, other patterned matrices with $O(n)$ entries, such as  reverse circulant and symmetric circulant matrices, also have unbounded support of the limiting spectral distribution. For a detailed and elementary study of this literature, one might look at the following nice book by Prof. Arup Bose \cite{bose2018patterned}. While D.B. was a master's student of Prof. Bose, he was asked an intriguing question by Prof. Bose which gave rise to this note. His question was whether all matrices formed by $O(n)$ many independent random variables have unbounded limiting spectral distribution if it exists. Considering this problem, we observed that the fundamental reason that existing patterned matrices have unbounded LSD is not because they have $O(n)$ many  independent random variables, but intrinsic patterns of the repetitions themselves. In particular, in a Hankel matrix, the same random variables appear on the lines parallel to cross diagonal and in a Toeplitz matrix, same random variables appear on the lines parallel to the diagonal. These structures give rise to higher values of the moments of the ESD. In contrast, if this structure is destroyed, then the empirical spectral distribution is again close to the semicircular law as stated in Corollary \ref{cor:conditional_lsd_X}.  
\section{The model}
Suppose, we have $n$ independent random variables $\xi_1, \ldots, \xi_n$ with $\bbE \xi_j:= 0$ and $\Var \xi_j = 1$ for $j \in [n]$. Moreover, assume that for $r \in \bbN$, $M_r:= \sup_{i} \bbE [|\xi_i|^r] < \infty$. Let $\cS_n$ denote the set of bijections of $[n]$. So, $|\cS_n| = n!$. Suppose $\Phi_1, \Phi_2, \ldots, \Phi_n$ are $n$-many $\iid$ samples from $\Unif(\cS_n)$. Now, for $i, j \in [n]$, define
\begin{equation}
    Y_{ij}:= \xi_{\Phi_i(j)}.
\end{equation}
Finally, define the $n \times n$ symmetric random matrix $X_n$ with entries
\begin{equation}
    X_{ij} := \frac{1}{\sqrt{2}} (Y_{ij} + Y_{ji}). 
\end{equation}
We are interested in the spectrum of the matrix $X$ as $n$ grows. Clearly, $Y_{ij}$'s are centered. Moreover, for $i, j, k, l \in [n]$, one gets that
\begin{align*}
    \Cov[Y_{ij}, Y_{kl}] &= \bbP [\Phi_i(j) = \Phi_k(l)]\\
    &= \delta_{ik} \delta_{jl} + \frac{1}{n} (1 - \delta_{ik})\\
    &= \frac{1}{n} - \frac{1}{n} \delta_{ik} + \delta_{ik} \delta_{jl}.  
\end{align*}
This immediately leads to
\begin{align*}
    \Var Y_{ij} &= 1\\
    \Cov[Y_{ij}, Y_{kl}] &= \begin{cases}
        \frac{1}{n} \qquad \text{if}\,\, i \ne k\\
        0 \qquad \text{if} \,\, i =k, j \ne l.
    \end{cases}
\end{align*}
Now note that $\bbE X_{ij} = 0$ and 
\begin{align*}
    \Cov[X_{ij}, X_{kl}] &= \frac{1}{2} \left( \Cov[Y_{ij}, Y_{kl}] + \Cov[Y_{ij}, Y_{lk}] + \Cov[Y_{ji}, Y_{kl}] + \Cov[Y_{ji}, Y_{lk}] \right) \\
    &= \frac{1}{2} \left( \frac{1}{n} - \frac{1}{n} \delta_{ik} + \delta_{ik} \delta_{jl} \right) + \frac{1}{2} \left( \frac{1}{n} - \frac{1}{n} \delta_{il} + \delta_{il} \delta_{jk} \right) \\
    &+ \frac{1}{2} \left( \frac{1}{n} - \frac{1}{n} \delta_{jk} + \delta_{jk} \delta_{il} \right) + \frac{1}{2} \left( \frac{1}{n} - \frac{1}{n} \delta_{jl} + \delta_{jl} \delta_{ik}\right) \\
    &= \frac{2}{n} - \frac{1}{2n} (\delta_{ik} + \delta_{il} + \delta_{jk} + \delta_{jl}) + \frac{1}{2} (\delta_{ik} \delta_{jl} + \delta_{il} \delta_{jk} + \delta_{jk} \delta_{il} + \delta_{jl} \delta_{ik}),
\end{align*}
in particular,
\begin{align*}
    \Var(X_{ij}) &= \left( 1 + \frac{1}{n} \right) + \left( 1 - \frac{1}{n} \right) \delta_{ij} \\
    &= \begin{cases}
        2 &\text{if}\,\, i = j, \\
        1 + \frac{1}{n} &\text{if}\,\, i \ne j.
    \end{cases}
\end{align*}
So, 
\begin{align*}
    \bbE \Tr(X_n^2) &= \sum_{i j } \Var(X_{ij}) \\
    &= 2n + (n^2 -n) \left( 1 + \frac{1}{n} \right) \\
    &= n^2 + 2n -1,
\end{align*}
which leads to
\[
    \lim_{n \to \infty} \frac{1}{n^2} \bbE \Tr(X_n ^2) = 1.
\]
So, we rescale $X_n$ by $\sqrt{n}$ and consider the matrix 
\[
    \tilde{X}_n := \frac{1}{\sqrt{n}} X_n.
\]

\section{Preliminaries}
In this section we introduce the necessary concepts of random matrix theory. Let $A$ be a symmetric $n \times n$ matrix. Denote its (real) eigenvalues by $\lambda_1(A), \ldots, \lambda_n(A)$ in non-increasing order. Denote by $\mu_A$ the \emph{spectral measure} of $A$, i.e.,
\[
    \mu_A := \frac{1}{n} \sum_{i=1}^n \delta_{\lambda_i(A)}.
\]
Now, assume that $A$ is a random matrix. Now, suppose $(A_n)_{n \in \bbN}$ is a sequence of random matrices. Then, we say $A_n$ converges weakly in probability to a measure $\mu$ if for all bounded continuous functions $f$ and $\varepsilon >0$,
\[
    \bbP \left[\left| \int f \,d\mu_{A_n} - \int f \, d\mu \right| > \varepsilon \right] \to 0.
\]
Then $\mu_A$ is a random measure. Define the \emph{expected empirical spectral measure} (EESD) of $A$, denoted by $\overline{\mu}_A$, through its action on test functions $\phi \in C_b(\bbR)$:
\[
    \int \phi(x) \, \overline{\mu}_A (dx) = \bbE \left[ \frac{1}{n} \sum_{i=1}^n \phi (\lambda_{i}(A) ) \right],
\]
where 
\[
    C_b(\bbR) := \{ \phi:\bbR \to \bbR\, \vert \, \phi\,\,\text{is bounded and continuous} \}.
\]
In this paper, we are interested in the LSD of the matrix $\tilde{X}_n$.
\section{Main results}
We are now ready to state our main result. 
\begin{theorem}[LSD of $\tilde{X}_n$] \label{thm:lsd_X}
    If all moments of $\xi_1$ are finite, then the ESD of $\tilde{X}_n$ converges weakly in probability to the standard semicircle law.
\end{theorem}
\begin{definition}[Bounded-Lipschitz distance]
    Let $\mu, \nu$ be two probability measures on $(\bbR, \mathcal{B}(\bbR))$. The bounded-Lipshcitz distance between $\mu$ and $\nu$ is defined as 
    \[
        d_\mathrm{BL}(\mu, \nu) := \sup_{\substack {\|f\|_\infty \le 1, \\ \mathrm{Lip}(f) \le 1}} \left| \int_\bbR f\, d\mu - \int_\bbR f \, d\nu \right|,
    \]
    where
    \begin{align*}
        \mathrm{Lip}(f) &:= \sup_{x \ne y} \frac{|f(x) - f(y)|}{d(x, y)},\\
        \|f\|_{\infty} &:= \sup_{x\in \bbR} |f(x)|.
    \end{align*}
\end{definition}
The bounded-Lipschitz metric metrizes the topology of weak convergence. In other words, Theorem~\ref{thm:lsd_X} is equivalent to:
\[
    d_{\mathrm{BL}} (\mu_{\tilde{X}}, \mu_\sc) \overset{\bbP}{\rightarrow} 0.
\]
This interpretation of Theorem~\ref{thm:lsd_X} leads to the following easy corollary:
\begin{corollary}[Conditional LSD]\label{cor:conditional_lsd_X}
    For all $\varepsilon, \delta, \eta > 0$, there exists $N \in \bbN$ and a sequence of sets $A_n \subseteq \cS_n^{\otimes n}$ with $ |A_n| > (1- \eta) (n!)^n$ for all $n \ge N$, such that 
    \[
        \bbP \left[d_\mathrm{BL} \left(\mu_{\tilde{X}_n}, \mu_\sc \right) > \varepsilon \bigg| \Phi_1= \phi_1, \ldots, \Phi_n = \phi_n \right] \le \delta,
    \]
    for all $(\phi_1, \ldots, \phi_n) \in A_n$.
\end{corollary}
The proofs of Theorem~\ref{thm:lsd_X} and Corollary~\ref{cor:conditional_lsd_X} are postponed to Section~\ref{proofs}.
\section{Simulations}
\begin{figure}[htbp]
    \centering
    \includegraphics[width=0.6\textwidth]{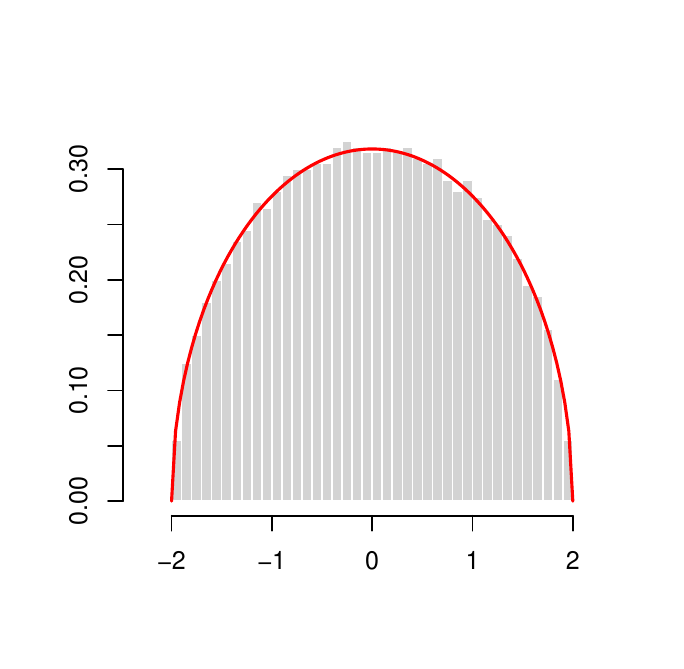}
    \caption{Empirical spectral distribution}
    \label{fig:eigenvalue-histogram}
\end{figure}
In figure~\ref{fig:eigenvalue-histogram}, we plot the histogram of the eigenvalues of the matrix $\tilde{X}_{n}$ for $n=2000$ with $\xi_i \overset{\iid}{\sim} \cN(0,1)$. The red curve is the density of the standard semicircle law which perfectly fits the histogram.
\section{Proofs}\label{proofs}
\begin{proof}[Proof of Theorem~\ref{thm:lsd_X}]
    The main ingredient of the proof is Theorem~5 of \cite{hochstattler2015semicircle}. By virtue of this theorem, it is enough to show that for $p, q \in \bbN$, distinct indices $i_1, \ldots, i_p, j_1, \ldots, j_p \in [n]$ and (possibly equal) indices $k_1, \ldots, k_q, l_1, \ldots, l_q \in [n]$ with $\{ i_1, \ldots, i_p, j_1, \ldots, j_p \} \cap \{ k_1, \ldots, k_q, l_1, \ldots, l_q \} = \emptyset$,
    \begin{align}
        &\bbE \left[ \prod_{u=1}^p X_{i_u j_u} \prod_{v=1}^q X_{k_v l_v} \right] \le \frac{C_{pq}}{n^{\frac{p}{2}}} \label{eq:cond_1}\\
        &\bbE \left[ \prod_{u=1}^p X_{i_u j_u}^2 \right] \to 1, \label{eq:cond_2}
    \end{align}
    for some ($n$-independent) constant $C_{pq}$. We prove them now.\\
    \textbf{Proof of \eqref{eq:cond_1}.}
    We denote by $\{\{ a_1, \ldots, a_m\}\}$ a multiset with (possibly repeated) elements $a_1, \ldots, a_m$. A multiset is called \emph{matched} if each element has multiplicity at least $2$. Define $\overline{\Phi}_1, \overline{\Phi}_2: [n]^2 \to [n]$ defined as
    \begin{align*}
        \overline{\Phi}_1 (i, j) &:= \Phi_i(j) \\
        \overline{\Phi}_2 (i, j) &:= \Phi_j(i)
    \end{align*}
    for $i, j \in [n]$. Clearly, 
    \[
        X_{i j} = \frac{1}{\sqrt{2}}\left( \xi_{\overline{\Phi}_1 (i,j)} + \xi_{\overline{\Phi}_2 (i,j)} \right).
    \]
    Now, notice that
    \begin{align*}
        & \bbE \left[ \prod_{u=1}^p X_{i_u j_u} \prod_{v=1}^q X_{k_v l_v} \right]\\
        &= \frac{1}{2^{\frac{p+q}{2}}} \sum_{\substack{t_1, \ldots, t_p=1 \\ s_1, \ldots, s_q = 1}}^2 \bbE \left[ \prod_{u=1}^p \xi_{\overline{\Phi}_{t_u}(i_u, j_u)} \prod_{v=1}^q \xi_{\overline{\Phi}_{s_v}(k_v, l_v)} \right].
    \end{align*}
    Here, we prove the desired bound for the case $t_1= \cdots = t_p = s_1 = \cdots = s_q =1$. The other cases will yield the same bound and the LHS of \eqref{eq:cond_1} will be bounded by $2^{\frac{p+q}{2}}$-times this bound.
    \begin{align*}
        & \bbE \left[ \prod_{u=1}^p \xi_{\overline{\Phi}_{1}(i_u, j_u)} \prod_{v=1}^q \xi_{\overline{\Phi}_{1}(k_v, l_v)} \right]\\
        &= \left| \bbE \left[ \prod_{u=1}^p \xi_{\Phi_{i_u}(j_u)} \prod_{v=1}^q \xi_{\Phi_{k_v}(l_v)} \right] \right| \\
        &= \left| \bbE \left[ \prod_{u=1}^p \xi_{\Phi_{i_u}(j_u)} \prod_{v=1}^q \xi_{\Phi_{k_v}(l_v)} \bbI(\{\{ \Phi_{i_1} (j_1), \ldots, \Phi_{i_p} (j_p), \Phi_{k_1}(l_1), \ldots, \Phi_{k_q}(l_q)\}\} \,\, \text{is matched}) \right] \right| \\
        &= \bigg| \bbE \bigg[ \bbI(\{\{ \Phi_{i_1} (j_1), \ldots, \Phi_{i_p} (j_p), \Phi_{k_1}(l_1), \ldots, \Phi_{k_q}(l_q)\}\} \,\, \text{is matched}) \\
        & \cdot \bbE \bigg[ \prod_{u=1}^p \xi_{\Phi_{i_u}(j_u)} \prod_{v=1}^q \xi_{\Phi_{k_v}(l_v)} \bigg| \Phi_{i_1}, \ldots, \Phi_{i_p}, \Phi_{k_1}, \ldots, \Phi_{k_q} \bigg]\bigg] \bigg|\\
        &\le \left( \prod_{u=1}^p \bbE[|\xi_{\Phi_{i_u}(j_u)}|^{p+q}] \prod_{v=1}^q \bbE[|\xi_{\Phi_{k_v}(l_v)}|^{p+q}] \right)^{\frac{1}{p+q}} \\
        &\cdot \bbP [\{\{ \Phi_{i_1} (j_1), \ldots, \Phi_{i_p} (j_p), \Phi_{k_1}(l_1), \ldots, \Phi_{k_q}(l_q)\}\} \,\, \text{is matched})] \qquad [\text{By H\"{o}lder's inequality}]\\
        &= M_{p+q} \cdot \bbP [\{\{ \Phi_{i_1} (j_1), \ldots, \Phi_{i_p} (j_p), \Phi_{k_1}(l_1), \ldots, \Phi_{k_q}(l_q)\}\} \,\, \text{is matched})] 
    \end{align*}
    Note that the number of partitions of the set $[p + q]$ is bounded by $e \cdot (p+q)^{p+q}$. Now fix some matched partition $A_1, \ldots, A_m$ of $[p+q]$. So, $|A_s| \ge 2$ for all $s$. Define $A^{(1)}_s:= [p] \cap A_s$ and $A^{(2)}_s := ([p+q] \setminus [p]) \cap A_s$. Also, define $r_s := |A^{(1)}_s|$. Then, $\sum_{s=1}^m r_s = p$. Suppose for each fixed $s \in [m]$, the numbers $\Phi_{i_u}(j_u), u \in A^{(1)}_s$ and $\Phi_{k_v}(l_v), v \in A^{(2)}_s$ are all equal. Notice that
    \begin{align*}
        &\bbP[\Phi_{i_u}(j_u), u \in A^{(1)}_s\,\, \text{and}\,\, \Phi_{k_v}(l_v), v \in A^{(2)}_s\,\, \text{are all equal}] \\
        &\le n^{- |\{i_u : u \in A^{(1)}_s\} \cup \{k_v: v \in A^{(2)}_s\}| + 1}\\
        &= n^{- |\{i_u : u \in A^{(1)}_s\}| - |\{k_v: v \in A^{(2)}_s\}| + 1} \qquad [\text{since}\,\, \{i_u: u \in [p]\} \cap \{k_v: v \in [q]\} = \emptyset]\\
        &= n^{- r_s - |\{k_v: v \in A^{(2)}_s\}| + 1}.
    \end{align*}
    Define
    \begin{align*}
        m_1 &:= | \{ s: A^{(2)}_s \ne \emptyset \}|, \\
        m_2 &:= | \{ s: A^{(2)}_s = \emptyset \}|.
    \end{align*}
    Clearly, $m = m_1 + m_2$. Also since the partition is paired, $m_2$ can be alternatively expressed as
    \[
        m_2 =  | \{ s: |A^{(1)}_s| \ge 2 \}|.   
    \]
    So, it follows that
    \[
        m_2 \le \frac{p}{2}.
    \]
    Now, by the union bound
    \begin{align*}
        &\bbP [\{\{ \Phi_{i_1} (j_1), \ldots, \Phi_{i_p} (j_p), \Phi_{k_1}(l_1), \ldots, \Phi_{k_q}(l_q)\}\} \,\, \text{is matched})] \\
        &\le e \cdot (p+q)^{p+q} \cdot \prod_{s=1}^m n^{- r_s - |\{k_v: v \in A^{(2)}_s\}| + 1}\\
        &= e \cdot (p+q)^{p+q} \cdot n^{ - \sum_{s=1}^m r_s - \sum_{s=1}^m |\{k_v: v \in A^{(2)}_s\}| + m} \\
        &= e \cdot (p+q)^{p+q} \cdot n^{-p - \sum_{s=1}^m |\{k_v: v \in A^{(2)}_s\}| + m} \\
        &\le e \cdot (p+q)^{p+q} \cdot n^{-p - m_1 + m} \qquad [\text{since}\,\, \sum_{s=1}^m |\{k_v: v \in A^{(2)}_s\}| \ge m_1] \\
        &= e \cdot (p+q)^{p+q} \cdot n^{-p + m_2} \\
        &\le \frac{e \cdot (p+q)^{p+q}}{n^{\frac{p}{2}}} \qquad \left[\text{since}\,\, m_2 \le \frac{p}{2} \right].
    \end{align*}
    This finishes the proof of \eqref{eq:cond_1}.\\
    \textbf{Proof of \eqref{eq:cond_2}.} Define the event 
    \[
        E:= \{\Phi_{i_u}(j_u) = \Phi_{i_v}(j_v)\,\, \text{for some}\,\, u \ne v\}.
    \]
    Then,  
    \begin{align*}
        \bbP[E] &\le \binom{p}{2} \cdot \bbP[\Phi_{i_1}(j_1) = \Phi_{i_2} (j_2)] \qquad [\text{by union bound}]\\
        &\le \binom{p}{2} \cdot \frac{1}{n}.  
    \end{align*}
    Define $\overline{\Phi}_3 : [n]^2 \to [n]$
    Now, notice that
    \begin{align*}
        & \bbE \left[ \prod_{u=1}^p X_{i_u j_u}^2 \right] \\
        &= \frac{1}{2^p} \bbE \left[ \prod_{u=1}^p \left( \xi_{\Phi_{i_u}(j_u)} + \xi_{\Phi_{j_u}(i_u)} \right)^2 \right] \\
        &= \frac{1}{2^p} \bbE \left[ \prod_{u=1}^p \left( \xi_{\Phi_{i_u}(j_u)}^2 + \xi_{\Phi_{j_u}(i_u)}^2 + 2 \xi_{\Phi_{i_u}(j_u)} \cdot \xi_{\Phi_{j_u}(i_u)}\right) \right] \\
        &= \frac{1}{2^{p}} \bbE \left[ \prod_{u=1}^p \left( \xi_{\overline{\Phi}_1 (i_u, j_u)} + \xi_{\overline{\Phi}_2 (i_u, j_u)} \right)^2 \right] + O \left( \frac{1}{\sqrt{n}} \right) \qquad[\text{by}\,\, \eqref{eq:cond_1}]\\
        &= \frac{1}{2^{p}} \bbE \left[ \prod_{u=1}^p \left( \xi_{\overline{\Phi}_1 (i_u, j_u)}^2 + \xi_{\overline{\Phi}_2 (i_u, j_u)}^2 \right) \right] + O \left( \frac{1}{\sqrt{n}} \right) \\
        &= \frac{1}{2^{p}} \sum_{t_1, \ldots, t_p =1}^2 \bbE \left[ \prod_{u=1}^p \xi_{\overline{\Phi}_{t_u} (i_u, j_u)}^2 \right] + O \left( \frac{1}{\sqrt{n}} \right).
    \end{align*}
    It is enough to investigate the case when $t_1 = \cdots = t_p = 1$, and prove that 
    \[
        \bbE \left[ \prod_{u=1}^p \xi_{\overline{\Phi}_1 (i_u, j_u)}^2 \right] \to 1.
    \]
    Similarly, the other terms yield the same limit. Now, 
    \begin{align*}
        &\bbE \left[ \prod_{u=1}^p \xi_{\overline{\Phi}_1(i_u, j_u)}^2 \right] \\
        &= \bbE \left[ \prod_{u=1}^p \xi_{\Phi_{i_u}(j_u)}^2 \right]\\
        &= \bbE \left[ \prod_{u=1}^p \xi_{\Phi_{i_u}(j_u)}^2 \bbI(E) \right] + \bbE \left[ \prod_{u=1}^p \xi_{\Phi_{i_u}(j_u)}^2 \bbI(E^c) \right]\\
        &= \bbE \left[ \prod_{u=1}^p \xi_{\Phi_{i_u}(j_u)}^2 \bbI(E) \right] + \bbP[E^c]\\
        &= 1 + \bbE \left[ \prod_{u=1}^p \xi_{\Phi_{i_u}(j_u)}^2 \bbI(E) \right] - \bbP[E].
    \end{align*}
    Now 
    \begin{align*}
        & \left|\bbE \left[ \prod_{u=1}^p X_{i_u j_u}^2 \right] - 1 \right| \\
        &\le \bbE \left[ \prod_{u=1}^p |\xi_{\Phi_{i_u}(j_u)}|^2 \bbI(E) \right] + \bbP[E]\\
        &\le \left( \prod_{u =1}^p \bbE[|\xi_{\Phi_{i_u}(j_u)}|^{2p}] \right)^{\frac{1}{p}} \cdot \bbP[E] + \bbP[E] \qquad [\text{by H\"{o}lder's inequality}]\\
        &\le (M_{2p} + 1) \cdot \binom{p}{2} \frac{1}{n} \to 0.
    \end{align*}
    Hence \eqref{eq:cond_2} is proved. This finishes the proof of Theorem~\ref{thm:lsd_X}.
\end{proof}
\begin{proof}[Proof of Corollary~\ref{cor:conditional_lsd_X}]
    From Theorem~\ref{thm:lsd_X}, it follows that 
    \[
        \bbP \left[ d_\mathrm{BL} \left( \mu_{\tilde{X}_n}, \mu_\sc \right) > \varepsilon \right] \to 0,
    \]
    as $n \to \infty$. By an application of Markov's inequality, it follows that
    \[
        \bbP \left[ \bbP \left[d_\mathrm{BL} \left(\mu_{\tilde{X}_n}, \mu_\sc \right) > \varepsilon \bigg| \Phi_1, \ldots, \Phi_n \right] >\delta \right] \le \frac{1}{\delta} \bbP \left[ d_\mathrm{BL} \left( \mu_{\tilde{X}_n}, \mu_\sc \right) > \varepsilon \right] \to 0.
    \]
    Hence, Corollary~\ref{cor:conditional_lsd_X} follows.
\end{proof}

\section*{Acknowledgements}
D.B. is grateful to Prof. Arup Bose for asking him this question.

\bibliographystyle{alpha}
\bibliography{main.bib}

\end{document}